\documentclass[12pt]{amsart}
\usepackage[stable]{footmisc}
\usepackage[latin1]{inputenc}
\usepackage{graphicx}
\usepackage{epstopdf}
\usepackage{amssymb,amsmath}
\usepackage{amsthm}
\usepackage{multirow}
\usepackage{wrapfig}
\usepackage{steinmetz}

\newtheorem{thm}{Theorem}

\theoremstyle{definition}

\theoremstyle{remark}

\newtheorem{case}{Case}

\title[On the degree-chromatic polynomial of a tree] {On the degree-chromatic \\polynomial of a tree} 
\author{Diego Cifuentes}
\date{\today}
\address{Diego Cifuentes, undergraduate student, Departamento de Matemáticas, Universidad de los Andes, Cra 1 No 18A 12, Bogotá, Colombia}
\email{df.cifuentes30@uniandes.edu.co}

\begin{document}

\begin{abstract}
The degree chromatic polynomial $P_m(G,k)$ of a graph $G$ counts the number of $k$-colorings in which no vertex has $m$ adjacent vertices of its same color. We prove Humpert and Martin's conjecture on the leading terms of the degree chromatic polynomial of a tree.
\end{abstract}

\maketitle

Given a graph $G$, Humpert and Martin defined its $m$-chromatic polynomial $P_m(G,k)$ to be the number of $k$-colorings of $G$ such that no vertex has $m$ adjacent vertices of its same color. They proved this is indeed a polynomial.  When $m=1$, we recover the usual chromatic polynomial of the graph $P(G,k)$ \cite{read_introduction_1968}. 

The chromatic polynomial is of the form $P(G,k)=k^n-e k^{n-1}+o(k^{n-1})$, where $n$ is the number of vertices and $e$ the number of edges of $G$. For $m>1$ the formula is no longer true, but Humpert and Martin conjectured the following formula which we now prove:

\begin{thm}[\cite{humpert_incidence_2010,humpert_incidence_2011}, Conjecture]\label{thm:main}
Let $T$ be a tree with $n$ vertices and let $m$ be an integer with $1<m<n$, then the following equation holds.

\begin{equation*}\label{eq:main}
P_m(T,k)=k^n - \displaystyle\sum_{v\in V(T)}{d(v)\choose m}k^{n-m} + o(k^{n-m})
\end{equation*}

\end{thm}

\begin{proof}
For a given coloring of $T$, say vertices $v_1$ and $v_2$ are ``friends'' if they are adjacent and have the same color. For each $v$, let $A_v$ be the set of colorings such that $v$ has at least $m$ friends. We want to find the number of colorings which are not in any $A_v$, and we will use the inclusion-exclusion principle:

\begin{equation*}
 P_m(T,k)= k^n - \sum_{v\in V} |A_v| + \sum_{v_1, v_2\in V} |A_{v_1} \cap A_{v_2}| - \ldots
\end{equation*}

We first show that $|A_v|={d(v)\choose m}k^{n-m}+o(k^{n-m})$. Let $A_v^{(l)}$ be the set of $k$-colorings such that $v$ has exactly $l$ friends. Then 

\begin{equation*}
|A_v|=\sum_{l=m}^{n-1}|A_v^{(l)}|=\sum_{l=m}^{n-1}{d(v)\choose l} k (k-1)^{n-l-1}={d(v)\choose m}k^{n-m}+o(k^{n-m})
\end{equation*}

To complete the proof, it is sufficient to see that for any set $S$ of at least 2 vertices $|\bigcap_{v\in S} A_v|=o(k^{n-m})$; clearly we may assume $S=\{v_1,v_2\}$.  Consider the following cases:

\begin{case}[$v_1$ and $v_2$ are not adjacent]
Split $A_{v_1}$ into equivalence classes

$$\sigma_1\sim \sigma_2 \Leftrightarrow \sigma_1(w)=\sigma_2(w) \textrm{ for all } w\neq v_2$$

 Note that each equivalence class $C$ consists of $k$ colorings, at most $\frac{d(v_2)}{m}$ of which are in $A_{v_2}$. Therefore

$$ |A_{v_1}\cap A_{v_2}|=\sum_C |C\cap A_{v_2}|\leq \sum_C \frac{d(v_2)}{m} = \frac{|A_{v_1}|}{k}\cdot \frac{d(v_2)}{m}$$

it follows that $\frac{|A_{v_1}\cap{A_{v_2}}|}{|A_{v_1}|}$ goes to $0$ as $k$ goes to infinity, so $|A_{v_1}\cap A_{v_2}|=o(k^{n-m})$.

\end{case}

\begin{case}[$v_1$ and $v_2$ are adjacent]
Let $W$ be the set of adjacent vertices to $v_2$ other than $v_1$. They are not adjacent to $v_1$ as $T$ has no cycles. Split $A_{v_1}$ into equivalence classes.

$$\sigma_1\sim \sigma_2 \Leftrightarrow \sigma_1(w)=\sigma_2(w) \textrm{ for all } w\notin W$$

Each equivalence class $C$ consists of $k^{|W|}$ colorings. If $v_1$ and $v_2$ are friends in the colorings of $C$, then a coloring in $|C \cap A_{v_2}|$ must contain at least $m-1$ vertices in $W$ of the same color as $v_2$. Therefore

$$|C\cap A_{v_2}| = \sum_{l=m-1}^{|W|}{|W| \choose l}(k-1)^{|W|-l} < \sum_{l=0}^{|W|}{|W| \choose l}k^{|W|-1}=2^{|W|}k^{|W|-1}$$

Notice that here we are using $m \geq 2$ so that $l \geq 1$. Otherwise, if $v_1$ and $v_2$ are not friends in the colorings of $C$, then 

$$|C\cap A_{v_2}| = \sum_{l=m}^{|W|}{|W| \choose l}(k-1)^{|W|-l} < \sum_{l=0}^{|W|}{|W| \choose l}k^{|W|-1}=2^{|W|}k^{|W|-1}$$

Therefore

\begin{align*}
|A_{v_1}\cap A_{v_2}|=\sum_C |C\cap A_{v_2}|<& \sum_C 2^{|W|}k^{|W|-1}\\
=& \frac{|A_{v_1}|}{k^{|W|}}\cdot 2^{|W|}k^{|W|-1}
=\frac{|A_{v_1}|\cdot 2^{|W|}}{k}
\end{align*}

and $|A_{v_1}\cap A_{v_2}|=o(k^{n-m})$ follows as in the first case.
\end{case}

This completes the proof of the theorem.
\end{proof}

\textbf{Acknowledgments.} I would like to thank Federico Ardila for bringing this problem to my attention, and for  helping me improve the presentation of this note. I would also like to acknowledge the support of the SFSU-Colombia Combinatorics Initiative.

\nocite{*}
\bibliographystyle{amsplain}
\bibliography{bib}






\end{document}